\documentclass[11pt,twoside]{amsart}
\usepackage{latexsym,amssymb,amsmath}

\numberwithin{equation}{section}
\newtheorem{theorem}{Theorem}[section]
\newtheorem{lemma}[theorem]{Lemma}
\newtheorem{proposition}[theorem]{Proposition}
\newtheorem{corollary}[theorem]{Corollary}

\theoremstyle{definition}
\newtheorem{definition}[theorem]{Definition} 
\newtheorem{remark}[theorem]{Remark}
\newtheorem{example}[theorem]{Example}
\begin{document}


\title[]{Computing the degree of a lattice ideal of dimension one}

\thanks{The first author was partially supported by CONACyT. 
The second author was partially supported by SNI}

\author{Hiram H. L\'opez}
\address{
Departamento de
Matem\'aticas\\
Centro de Investigaci\'on y de Estudios
Avanzados del
IPN\\
Apartado Postal
14--740 \\
07000 Mexico City, D.F.
}

\author{Rafael H. Villarreal}
\address{
Departamento de
Matem\'aticas\\
Centro de Investigaci\'on y de Estudios
Avanzados del
IPN\\
Apartado Postal
14--740 \\
07000 Mexico City, D.F.
}
\email{vila@math.cinvestav.mx}


\begin{abstract} We show that the degree of a graded lattice ideal of
dimension $1$ is the order of the torsion subgroup of the
quotient group of the lattice. This gives an efficient method to
compute the degree of this type of lattice ideals. 
\end{abstract}

\maketitle 

\section{Introduction}\label{intro-degree-lattice}

Let $S=K[t_1,\ldots,t_s]$ be a graded polynomial ring over a
field $K$, where each $t_i$ is homogeneous of degree one, 
and let $S_d$ denote the set of homogeneous polynomials of
total degree $d$ in $S$, together with the zero polynomial. The set
$S_d$ is a $K$-vector space of dimension $\binom{d+s-1}{s-1}$. If
$I\subset S$ is a graded ideal, i.e., $I$ is generated by homogeneous
polynomials, we let 
$$
I_d=I\cap S_d,
$$
denote the set of homogeneous polynomials in $I$ of total degree $d$,
together with the zero polynomial. 
Note that $I_d$ is a vector
subspace of $S_d$. Then the {\it Hilbert function\/} of the quotient
ring $S/I$, denoted by $H_I(d)$, is defined by
$$ 
H_I(d)=\dim_K(S_d/I_d). 
$$
According to a classical result of Hilbert
\cite[Theorem~4.1.3]{BHer}, there is a unique polynomial 
\begin{eqnarray*}  
h_I(t)=c_kt^k+(\mbox{terms of lower degree})
\end{eqnarray*}
of degree $k$, with rational coefficients, such that $h_I(d)=H_I(d)$
for $d\gg 0$. By 
convention the zero polynomial has degree $-1$, that is, $h_I(t)=0$ if
and only if $k=-1$. The integer $k+1$ is 
the {\it Krull dimension\/} of $S/I$ and $h_I(t)$ is the {\it Hilbert
polynomial\/} 
of $S/I$. If $k\geq 0$, the positive integer $c_{k}(k!)$ is called 
the {\it degree\/} of $S/I$. The {\it degree\/} of $S/I$ is 
defined as $\dim_K(S/I)$ if $k=-1$. The {\it index of regularity\/}
of $S/I$, denoted by  
${\rm reg}(S/I)$, is the least integer $r\geq 0$ such that
$h_I(d)=H_I(d)$ for $d\geq r$. The degree and the Krull dimension are
denoted  
by ${\rm deg}(S/I)$ and $\dim(S/I)$, respectively. As usual, by
the {\it dimension\/} of $I$ we mean the Krull dimension of $S/I$.

The notion of {\it degree\/} plays an important role in algebraic
geometry  \cite{CLO,cox-toric,harris} and commutative algebra
\cite{BHer,Eisen}. Consider a projective space $\mathbb{P}^{s-1}$
over the field $K$.  
The {\it degree\/} and the {\it
dimension\/} of a projective variety $X\subset\mathbb{P}^{s-1}$ can
be read off 
from  the Hilbert polynomial $h_I(t)$, where
$I=I(X)$ is the {\it vanishing ideal\/} of $X$ generated by the
homogeneous 
polynomials  of $S$ that vanish at all points of $X$. For the
geometric interpretation of the degree 
of $S/I(X)$ see the references above. If $X$ is a finite set
of points, the Hilbert polynomial of $S/I(X)$ is a non-zero constant,  
the degree of $S/I(X)$ is equal to $|X|$ (the number of points in
$X$), and the dimension of $S/I(X)$ is equal to $1$
\cite[p.~164]{harris}. In view of its applications to coding theory, 
we are interested in the case when $K$ is a
finite field and $X$ is parameterized by monomials 
(see Section~\ref{vanishing-id-section}). 

Let $\mathcal{L}\subset \mathbb{Z}^s$ be a {\it lattice\/}, i.e., 
$\mathcal{L}$ is a subgroup of $\mathbb{Z}^s$. The {\it lattice
ideal\/} of $\mathcal{L}$, denoted by $I(\mathcal{L})$,  
is the ideal of $S$ generated by the set of all 
binomials $t^{a^+}-t^{a^-}$ such that $a\in\mathcal{L}$, where $a^+$
and $a^-$ 
are the positive and negative part of $a$
(see Section~\ref{degree-lattice-section}).  
A first hint of the rich interaction between the group theory of
$\mathcal{L}$ and the algebra of $I(\mathcal{L})$ is that the 
rank of $\mathcal{L}$, as a free abelian group, is equal to
$s-\dim(S/I(\mathcal{L}))$ \cite[Proposition~7.5]{cca}. This number 
is the height of the ideal $I(\mathcal{L})$ in the sense of
\cite{BHer}, and is usually denoted by
${\rm ht}(I(\mathcal{L}))$. Another useful result is that
$\mathbb{Z}^s/\mathcal{L}$ is a torsion-free group if and only if
$I(\mathcal{L})$ is a 
prime ideal \cite[Theorem~ 7.4]{cca}. In the same vein, for a certain
family of lattice ideals, we will relate the structure of the 
finitely generated abelian group $\mathbb{Z}^s/\mathcal{L}$ and the
degree of $S/I(\mathcal{L})$. 

The set of nonnegative integers (resp. positive integers) 
is denoted by $\mathbb{N}$ (resp. $\mathbb{N}_+$).
The structure of $\mathbb{Z}^s/\mathcal{L}$ can easily be 
determined, as we now explain. Let $A$ be an integral matrix of order
$m\times s$
whose rows generate $\mathcal{L}$. There are unimodular integral
matrices $U$ and $V$ such that $UAV=D$, where  
$D={\rm diag}(d_1,d_2,\ldots
, d_r,0,\ldots ,0)$ is a diagonal matrix, with $d_i\in\mathbb{N}_+$
and $d_i$ divides $d_j$ if 
$i\leq j$. The matrix $D$ is the {\it Smith normal form\/} of
$A$ and the integers
$d_1,\ldots,d_r$ are the {\it invariant factors} of $A$. Recall that
the {\it torsion subgroup\/} of $\mathbb{Z}^s/\mathcal{L}$, denoted by 
$T(\mathbb{Z}^s/\mathcal{L})$, consists of all
$\overline{a}\in\mathbb{Z}^s/\mathcal{L}$ such that
${\ell}\,\overline{a}=\overline{0}$ for some $\ell\in\mathbb{N}_+$. 
From the fundamental structure theorem for finitely 
generated abelian groups \cite{JacI} one has: 
\begin{eqnarray*}
& &\mathbb{Z}^s/\mathcal{L}\simeq \mathbb{Z}/(d_1)\oplus
\mathbb{Z}/(d_2)\oplus\cdots\oplus\mathbb{Z}/(d_r)\oplus\mathbb{Z}^{s-r},\\
&&\ \ \ \ \ \ \ \  \ \ \ \ \ \ \ \  T(\mathbb{Z}^s/\mathcal{L})\simeq
\mathbb{Z}/(d_1)\oplus 
\mathbb{Z}/(d_2)\oplus\cdots\oplus\mathbb{Z}/(d_r),
\end{eqnarray*}
where $r$ is the rank of $\mathcal{L}$. In particular the order of 
$T(\mathbb{Z}^s/\mathcal{L})$ is $d_1\cdots d_r$. Thus, using any algebraic
system that compute Smith normal forms of integral matrices, {\it
Maple\/} \cite{maple} for instance,  
one can determine the order of $T(\mathbb{Z}^s/\mathcal{L})$.

Let $I(\mathcal{L})\subset S$ be a graded lattice ideal 
of dimension one. Note that the
corresponding lattice $\mathcal{L}$ is homogeneous and has rank $s-1$ (see
Definition~\ref{homogeneous-lattice-def} and
Remark~\ref{remark-lh-rh}).
 The aim of this paper is to give a new method, 
using integer linear algebra, to compute the degree of
$S/I(\mathcal{L})$.

The contents of this paper are as follows. 
In Section~\ref{hilbert-function-section}, we present some well known results
about the behavior of Hilbert functions of graded ideals. In
particular, we recall a standard method, using Hilbert series, 
to compute the degree and the index of regularity.  

In Section~\ref{degree-lattice-section}, we use linear algebra and
Gr\"obner bases methods to describe 
the torsion subgroup of $\mathbb{Z}^s/\mathcal{L}$ (see
Lemmas~\ref{may5-12} and \ref{may6-12}). 
Then, using standard Hilbert functions techniques, we give an upper
bound for the index of 
regularity (see Proposition~\ref{may6-12-1}). 

The main result of this
paper is the following formula for the degree:
$$\deg\, S/I(\mathcal{L})=|T(\mathbb{Z}^s/\mathcal{L})|,$$
where $|T(\mathbb{Z}^s/\mathcal{L})|$ is the cardinality  of
the torsion subgroup (see Theorem~\ref{degree-lattice}). As a
consequence, if $\mathcal{L}$ is generated as a $\mathbb{Z}$-module by
the rows of an integral matrix
$A$, then
$$
{\rm deg}\, S/I(\mathcal{L})=d_1\cdots d_{s-1},
$$
where $d_1,\ldots,d_{s-1}$ are the invariant factors of $A$ (see
Corollary~\ref{may28-12}). This gives a method to compute the degree
directly from a set of
generators of the lattice using the Smith normal form from linear
algebra. It would be interesting to compute the index of 
regularity in terms of the lattice using linear algebra
methods. Some other problems for future works are included in
Section~\ref{future-works-section}.  

If $\mathfrak{B}$ is a $\mathbb{Z}$-basis for $\mathcal{L}$
and $\mathcal{P}$ is the convex hull of $\mathfrak{B}\cup\{0\}$ in
$\mathbb{R}^s$, we obtain the following expression for the
degree:  
$$
{\rm deg}\, S/I(\mathcal{L})=(s-1)!{\rm vol}(\mathcal{P}),
$$
where ``${\rm vol}$'' is the relative volume in the sense 
of \cite{ehrhart,Sta5} (see
Corollary~\ref{jun2-12}). If an
integral polytope in $\mathbb{R}^s$ has dimension $s$, then its relative 
volume agrees with its usual volume \cite[p.~239]{Sta5}. Note that in
our situation the usual volume of $\mathcal{P}$ is $0$ because
$\mathcal{P}$ has dimension $s-1$. This is why we express the degree 
in terms of the relative volume.

There are standard methods to compute the degree of 
any graded lattice ideal, using Gr\"obner
bases and Hilbert series (see Sections~\ref{hilbert-function-section}
and \ref{section-examples}), but our method is far more efficient, 
especially with large examples (see Examples~\ref{jun2-12-1} and
\ref{jun2-12-2}). Our main result cannot be generalized to arbitrary
lattice ideals (see Example~\ref{jun4-12}).

In Section~\ref{vanishing-id-section}, we consider the case when $K$
is a finite field and $X$ is a subset of the projective space 
$\mathbb{P}^{s-1}$ over the field $K$. If $X$ is parameterized by
monomials, by Theorem~\ref{lemma-homogeneous-lattice}, our results
can be applied 
to the vanishing ideal $I(X)$. The degree of a vanishing ideal is
relevant from the 
viewpoint of coding theory because it occurs as one of the
main parameters of evaluation codes \cite{algcodes}.  We show an
instance where the algebraic 
structure of the lattice is reflected in the algebraic
structure of the vanishing ideal (see Corollary~\ref{jun2-12-3}).  

For all unexplained
terminology and additional information,  we refer to 
\cite{EisStu,cca,monalg} (for the theory of binomial and lattice
ideals) and \cite{CLO,Sta1} (for Gr\"obner bases and Hilbert
functions).

\section{Hilbert functions and the degree}\label{hilbert-function-section}

In this section we recall some well known results about the behavior
of Hilbert 
functions of graded ideals and recall a standard
method, using Hilbert series, to compute the degree.  

We continue to use the notation and definitions used in
Section~\ref{intro-degree-lattice}. 
Let 
$$S=K[t_1,\ldots,t_s]=
\bigoplus_{d=0}^\infty S_d
$$ 
be a graded polynomial ring over a
field $K$, with the grading induced by setting $\deg(t_i)=1$
for $i=1,\ldots,s$, and let $I\subset S$ be a graded
ideal. The ring $S/I$ inherits a graded structure 
\[
S/I=\bigoplus_{d=0}^\infty (S/I)_d
\]
whose $d$-th component is given by $(S/I)_d=S_d/I_d$, where 
$I_d=I\cap S_d$. Recall that the Hilbert function of $S/I$ is given by 
$$H_I(d)=\dim_K
S_d/I_d.$$
 
The degree and the regularity of $S/I$ can be computed using Hilbert
series, as
we now explain. By the Hilbert Serre theorem, there is a unique
polynomial $g(t)\in\mathbb{Z}[t]$ with $g(1)\neq 0$ such that 
the {\it Hilbert series} $F_I(t)$ of
$S/I$ can be 
written as
\begin{equation*}
F_I(t):=\sum_{d=0}^{\infty}H_I(d)t^d=\frac{g(t)}{(1-t)^\lambda},
\end{equation*}
where $\lambda$ is the Krull dimension of $S/I$. The degree of $S/I$ is
equal to $g(1)$ and the index of regularity of $S/I$ is equal to $0$ if
$\deg(g(t))-\lambda<0$ and is equal to $\deg(g(t))-\lambda+1$
otherwise. Thus,
the computation of the degree is reduced to the 
computation of the Hilbert series of $S/I$. There are a number of
computer algebra  
systems ({\it Macaulay\/}$2$ \cite{mac2}, {\it CoCoA}, {\it
Singular}) that compute the Hilbert series and the degree of $S/I$ using
Gr\"obner bases. Two excellent references for computing
Hilbert series, using elimination of variables, are \cite{BS1,Bigatti1}.
For toric ideals there are  methods, implemented in {\it Normaliz\/}
\cite{normaliz2}, to compute its Hilbert series and its degree 
using polyhedral geometry.  

We could not find a reference for the next simple lemma. 

\begin{lemma}\label{nov6-11-grace} {\rm(a)} If $S_i=I_i$ for some
$i\geq 1$, then $S_d=I_d$ for all $d\geq i$.

{\rm(b)}  If $\dim\, S/I\geq 1$, then $H_I(i)>0$ 
for $i\geq 0$.  
\end{lemma}
\begin{proof} (a) It suffices to prove the 
case $d=i+1$. As $I_{i+1}\subset S_{i+1}$, we need only show
$S_{i+1}\subset I_{i+1}$. Take 
a monomial $f$ in $S_{i+1}$. Then, 
$f=t_1^{a_1}\cdots t_{s}^{a_s}$ with $\sum_{i=1}^sa_i=i+1$ and
$a_j>0$ for some 
$j$. Thus, $f\in S_1S_{i}$. As $S_1S_i=S_1I_i\subset I_{i+1}$, 
we get $f\in I_{i+1}$. 

(b) The Hilbert polynomial $h_I$ of $S/I$ has degree $\dim(S/I)-1\geq 0$.
Hence, $h_I$ is a non-zero polynomial. If $H_I(i)=\dim_K(S/I)_i=0$
for some $i$, then $S_i=I_i$. Thus, by (a), 
$H_I(d)$ vanishes for $d\geq i$, a 
contradiction because the Hilbert polynomial of $S/I$ is non-zero.  
\end{proof}

Next, we recall and prove a general fact about $1$-dimensional
Cohen-Macaulay 
graded ideals:  {\it the Hilbert function is increasing 
until it reaches a constant value\/}. This behaviour was pointed
out in \cite[p.~456]{duursma-renteria-tapia} 
(resp. \cite[Remark~1.1, p.~166]{geramita-cayley-bacharach}) for
finite (resp. infinite) 
fields, see also \cite{DGM}. No proof was given in neither of these
places, likely because 
the result is not hard to show. 

\begin{proposition}
\label{hilbert-function-dim=1} {\rm(i)} If $\dim\, S/I\geq 2$ and
${\rm depth}\, S/I>0$, 
then $H_I(i)<H_I(i+1)$ for
$i\geq 0$.

{\rm(ii)} If ${\rm depth}\, S/I=\dim\, S/I=1$, then there is an
integer $r\geq 0$ such that
$$
1=H_I(0)<H_I(1)<\cdots<H_I(r-1)<H_I(i)=\deg(S/I)\hspace{.5cm} {\it for}
\ i\geq r.
$$
\end{proposition}
\begin{proof} Consider the algebraic
closure $\overline{K}$ of the field $K$. We set
$$
\overline{S}=S\otimes_K \overline{K}\ \mbox{ and }\ \overline{I}=I
\overline{S}.
$$
By \cite[Lemma 1.1]{Sta1}, $S/I$ and $\overline{S}/\overline{I}$ have
the same Krull  
dimension, the same depth, and the same Hilbert function. 
This follows by considering the minimal free resolution
$\mathbf{F}_\star$ of $S/I$ and observing that
$\mathbf{F}_\star\otimes_K\overline{K}$ is the minimal free resolution
of $\overline{S}/\overline{I}$ and has the same numerical data
(Betti numbers and shifts) as $\mathbf{F}_\star$. Hence,
replacing $K$ by $\overline{K}$, we
may assume that $K$ is  
infinite. As $S/I$ has positive depth, there is $h\in S_1$ which is a
non zero-divisor of
$S/I$. Applying the function $\dim_K(\cdot)$ to the exact sequence 
$$
0\longrightarrow (S/I)[-1] \stackrel{h}{\longrightarrow}
S/I \longrightarrow S/(h,I)\longrightarrow 0,
$$
we get $H_I(i+1)-H_I(i)=H(i+1)\geq 0$ for $i\geq 0$, where
$H(i)=\dim_K(S/(h,I))_i$. We set $S'=S/(h,I)$. Notice that
$\dim(S')=\dim(S/I)-1$. 

(i) If $H(i+1)=0$ for some $i\geq 0$, then, by
Lemma~\ref{nov6-11-grace}(a), $\dim_K(S')<\infty$. Hence 
$S'$ is Artinian, i.e., $\dim(S')=0$, a
contradiction. Thus, $H_I(i+1)>H_I(i)$ for $i\geq 0$.

(ii) Since $\dim(S/I)=1$, the Hilbert polynomial of $S/I$ is a 
non-zero constant equal to ${\rm deg}(S/I)$. Let $r\geq 0$ be the
first integer such that 
$H_I(r)=H_I(r+1)$, 
thus $S'_{r+1}=(0)$, i.e., $S_{r+1}=(h,I)_{r+1}$. 
Then, by Lemma~\ref{nov6-11-grace}(a), $S'_i=(0)$ for $i\geq r+1$. Hence,
the Hilbert function of $S/I$ is constant for $i\geq r$ 
and strictly increasing on
$[0,r-1]$. 
\end{proof}

\section{The degree of a lattice ring}\label{degree-lattice-section}

We continue to use the notation and definitions used in
Sections~\ref{intro-degree-lattice} and
\ref{hilbert-function-section}. 
Given a homogeneous lattice
$\mathcal{L}\subset\mathbb{Z}^s$ of rank $s-1$, in this section we describe
the torsion subgroup of $\mathbb{Z}^s/\mathcal{L}$ using linear
algebra and Gr\"obner bases techniques. Then, we show that the
degree of the lattice ring $S/I(\mathcal{L})$ is
the order of the torsion subgroup of
$\mathbb{Z}^s/\mathcal{L}$. 

Recall that a binomial in $S$ is a polynomial of the
form $t^a-t^b$, where 
$a,b\in \mathbb{N}^s$ and where, if
\mbox{$a=(a_1,\dots,a_s)\in\mathbb{N}^s$}, we set 
\[
t^a=t_1^{a_1}\cdots t_s^{a_s}\in S. 
\]
A polynomial of the form $t^a-t^b$ is usually referred 
to as a {\it pure binomial\/} \cite{EisStu}, although here we are dropping the
adjective ``pure''.  A {\it binomial ideal\/} is an ideal generated by binomials. 

Given $a=(a_i)\in {\mathbb Z}^s$, the set ${\rm supp}(a)=\{i\, |\,
a_i\neq 0\}$ is called the {\it support\/} of
$a$. The vector $a$ can be written uniquely as $a=a^+-a^-$, 
where $a^+$ and $a^-$ are two nonnegative vectors 
with disjoint support, the {\it positive\/} and 
the {\it negative\/} part of $a$ respectively. 

\begin{definition}\label{lattice-ideal-def}\rm 
Let $\mathcal{L}\subset \mathbb{Z}^s$ be a {\it lattice\/}, that is, 
$\mathcal{L}$ is a subgroup of $\mathbb{Z}^s$. The {\it lattice
ideal\/} of $\mathcal{L}$ is the binomial ideal
$$
I(\mathcal{L}):=(\{t^{a^+}-t^{a^-}\vert\, 
a\in\mathcal{L}\})\subset S.
$$
The {\it lattice ring\/} of $\mathcal{L}$ is the quotient ring
$S/I(\mathcal{L})$.
\end{definition}

Lattice ideals have been studied
extensively, see  \cite{EisStu,eto,ci-lattice,cca,morales-thoma} and the
references therein. The concept of a lattice ideal is a natural
generalization of a toric ideal 
\cite[Corollary~7.1.4]{monalg}. A lattice ideal $I(\mathcal{L})$ is a
toric ideal if and only if $\mathbb{Z}^s/\mathcal{L}$ is
torsion-free \cite[Theorem~ 7.4]{cca}. 

The next lemma gives the general form of the binomials in a
lattice ideal. 

\begin{lemma}\label{aug-25-13} Let 
$\mathcal{L}\subset\mathbb{Z}^s$ be a lattice and let $t^a-t^b$ be a
binomial in $S$. Then, $a-b$ is in $\mathcal{L}$ if and only if
$t^a-t^b$ 
 is in $I(\mathcal{L})$. 
\end{lemma}

\begin{proof} If $a-b\in\mathcal{L}$, then by definition of a
lattice ideal one has $t^a-t^b\in I(\mathcal{L})$. To show the
converse assume that $t^a-t^b$ is in $I(\mathcal{L})$. We 
define an equivalence relation $\sim_\mathcal{L}$ on the set of
monomials of $S$ by 
$t^{c}\sim_\mathcal{L} t^{d}$ if and only if $c-d\in \mathcal{L}$.
Since $t^{a}$ and $t^{b}$ are not in
$I(\mathcal{L})$, by \cite[Lemma~2.2]{ElVi}, we get 
that $t^{a}\sim_\mathcal{L}
t^{b}$. Thus, $a-b\in\mathcal{L}$. 
\end{proof}

Given a binomial $g=t^a-t^b$, we set $\widehat{g}=a-b$. If $B$ is a
subset of $\mathbb{Z}^s$, $\langle B\rangle$ denotes the subgroup of 
$\mathbb{Z}^s$ generated by $B$. 

A lattice ideal is defined by a unique lattice.  

\begin{lemma}{\cite{ci-lattice}}\label{dec21-11} Let
$\mathcal{L}\subset\mathbb{Z}^s$ be 
a lattice and let $I(\mathcal{L})$ be its lattice ideal. If
$g_1,\ldots,g_r$ is a set of binomials that generate
$I(\mathcal{L})$, then
$\mathcal{L}=\langle\widehat{g}_1,\ldots,\widehat{g}_r\rangle$. In
particular if $L$ is a lattice ideal, there is a unique lattice
$\mathcal{L}$ such that $L=I(\mathcal{L})$.
\end{lemma}

The following is a well known description of lattice
ideals that follows from \cite[Corollary~2.5]{EisStu}.

\begin{theorem}{\rm\cite{EisStu}}\label{jun12-02} If $L$ is a binomial ideal 
of $S$, then $L$ is a lattice ideal if and only if $t_i$ is a non-zero
divisor of $S/L$ 
for all $i$. 
\end{theorem}

\begin{proof} Let $g_1,\ldots,g_r$ be a set of generators of $L$
consisting of binomials. Since $L$ does not contain any
monomials, by \cite[Corollary~2.5]{EisStu}, there is a unique 
lattice $\mathcal{D}\subset\mathbb{Z}^s$ such that 
\begin{equation}\label{23-08-13}
(L\colon h^\infty)=I(\mathcal{D}),\tag{$\dag$}
\end{equation}
where $h=t_1\cdots t_s$ and 
$(L\colon h^{\infty}):=\{f\in S\vert\, fh^m\in L\mbox{ for some
}m\geq 1\}$. 
It is not hard to see that 
$\mathcal{D}$ is generated by
$\widehat{g}_1,\ldots,\widehat{g}_r$ 
(cf. Proposition~\ref{saturation-gives-lattice}).   

$\Rightarrow$) If $L=I(\mathcal{L})$ for some lattice $\mathcal{L}$,
by Lemma~\ref{dec21-11}, $\mathcal{L}=\mathcal{D}$. 
Thus, $(L\colon h^\infty)=L$. To show that $t_i$ is a non 
zero-divisor of $S/L$ assume that there is $\overline{f}\in S/L$ 
such that $t_i\overline{f}=\overline{0}$. Thus, $t_if\in L$ and $hf\in
L$. Hence $f\in L$, i.e., $\overline{f}=\overline{0}$, as required.

$\Leftarrow$) By Eq.~(\ref{23-08-13}), we need only show the equality  
$(L\colon h^\infty)=L$. The inclusion $(L\colon h^\infty)\supset L$ is
clear. To show the reverse inclusion take $f\in (L\colon h^\infty)$,
i.e., $h^mf\in L$. Since $t_i$ is a nonzero divisor for all $i$, we
get that $f\in L$ as required. 
\end{proof}

The unique lattice that defines a graded lattice ideal is homogeneous
in the following sense.  

\begin{definition}\label{homogeneous-lattice-def} If
$a=(a_1,\ldots,a_s)\in\mathbb{Z}^s$, we set 
$|a|=\sum_{i=1}^sa_i$. A lattice $\mathcal{L}$ is called {\it
homogeneous\/} if $|a|=0$ for all $a\in\mathcal{L}$. 
\end{definition}

\begin{remark}\label{remark-lh-rh} (i) A lattice is homogeneous if
and only if its lattice 
ideal is graded. This follows from Lemma~\ref{dec21-11}. 

(ii) If $\mathcal{L}$ is a homogeneous lattice in $\mathbb{Z}^s$ of
rank $s-1$, then $S/I(\mathcal{L})$ is a Cohen-Macaulay ring of
dimension $1$. This follows from Theorem~\ref{jun12-02} and using
the fact that the height of $I(\mathcal{L})$ is the rank of
$\mathcal{L}$ \cite[Proposition~7.5]{cca}.
\end{remark}

\begin{definition}\rm The {\it torsion subgroup\/} of an abelian group
$(M,+)$, denoted by $T(M)$, is the set of all $x$ in $M$ such that $px=0$
for some $0\neq p\in\mathbb{N}$. 
\end{definition}

Next, we determine a generating set for the torsion subgroup 
of $\mathbb{Z}^s/\mathcal{L}$. 

\begin{lemma}\label{may5-12} Let $\mathcal{L}\subset\mathbb{Z}^s$ be
a homogeneous lattice of 
rank $s-1$ and let $\mathbb{Q}\mathcal{L}$ be the $\mathbb{Q}$-linear
space spanned 
by $\mathcal{L}$. Then  
\begin{itemize}
\item[(a)]
$\mathbb{Q}\mathcal{L}\cap\mathbb{Z}^s=
\mathbb{Z}(e_1-e_s)\oplus\cdots\oplus\mathbb{Z}(e_{s-1}-e_s)$, where
$e_i$ is the $i^{\textup{th}}$ unit vector in $\mathbb{Q}^s$.
\item[(b)]
$T(\mathbb{Z}^s/\mathcal{L})=
\mathbb{Z}(e_1-e_s)\oplus\cdots\oplus\mathbb{Z}(e_{s-1}-e_s)/\mathcal{L}$.
\end{itemize}
\end{lemma}

\begin{proof} (a) ``$\subset$'': Take $a=(a_1,\ldots,a_s)$ in
$\mathbb{Q}\mathcal{L}\cap\mathbb{Z}^s$. Then,
$a_s=-a_1-\cdots-a_{s-1}$ and 
we can write 
$$
a=a_1(e_1-e_s)+\cdots+a_{s-1}(e_{s-1}-e_s).
$$
Thus, $a$ is a $\mathbb{Z}$-linear combination of
$e_1-e_s,\ldots,e_{s-1}-e_s$. 

``$\supset$'': It suffices to show that
$e_k-e_s$ is in $\mathbb{Q}\mathcal{L}$ for all $k$. 
The dimension of $\mathbb{Q}\mathcal{L}$ is equal to 
${\rm rank}(\mathcal{L})=s-1$. Notice that
$e_s\notin \mathbb{Q}\mathcal{L}$ 
because $\mathcal{L}$ is homogeneous. Hence, 
$\mathbb{Q}e_s+\mathbb{Q}\mathcal{L}=\mathbb{Q}^s$. Therefore, we can write
$$
e_k=\mu_{ks}e_s+\lambda_{k1}\alpha_{1}+\cdots+\lambda_{km}\alpha_m\ \ \
(\mu_{ks}\in\mathbb{Q};\ \lambda_{ki}\in\mathbb{Q};\
\alpha_j\in\mathcal{L}\mbox{ for all }i,j). 
$$
Hence, taking inner products with $\mathbf{1}=(1,\ldots,1)$ and using
that $\langle\mathbf{1},\alpha_i\rangle=0$ for all $i$, we get
$\mu_{ks}=1$. Thus, $e_k-e_s\in \mathbb{Q}\mathcal{L}$.

(b): By \cite[Lemma~2.3]{ehrhart}, the torsion subgroup of
$\mathbb{Z}^s/\mathcal{L}$ is
$\mathbb{Q}\mathcal{L}\cap\mathbb{Z}^s/\mathcal{L}$. Hence, the
expression for the torsion follows from (a).
\end{proof}

In what follows we shall assume that $\succ$ is the {\it reverse
lexicographical order\/} (revlex order for short) on the
monomials of $S$. This order is given by
$t^b\succ t^a$ if and only if the last non-zero
entry  of $b-a$ is negative. As usual, if $g$ is a
polynomial of $S$, we denote the leading term of $g$ 
by ${\rm in}(g)$. If $L$ is an ideal of $S$, the initial
ideal of $L$, denoted by ${\rm in}(L)$, 
is generated by the leading terms of the polynomials of $L$.  

\begin{remark}\label{jun7-12}
By Buchberger's  algorithm \cite[Theorem~2, p.~89]{CLO} and
\cite[Proposition~6, p.~91]{CLO}, a graded lattice 
ideal $I(\mathcal{L})$ has a unique reduced Gr\"obner basis
$\mathcal{G}$ consisting of 
homogeneous binomials and, by Theorem~\ref{jun12-02}, 
each binomial $t^a-t^b\in\mathcal{G}$
satisfies that ${\rm supp}(a)\cap{\rm supp}(b)=\emptyset$.
\end{remark}

\begin{lemma}\label{may6-12} Let $\mathcal{L}\subset\mathbb{Z}^s$ be
a homogeneous lattice of 
rank $s-1$. Then, given 
$\widetilde{\alpha}=\alpha+\mathcal{L}$ in the torsion subgroup
$T(\mathbb{Z}^s/\mathcal{L})$ there 
exists a unique $a=(a_1,\ldots,a_{s-1},a_s)$ in $\mathbb{Z}^s$ such
that  
\begin{itemize}
\item[(i)] $a_i\geq 0$ for $i=1,\ldots,s-1$,
\item[(ii)] $t_1^{a_1}\cdots t_{s-1}^{a_{s-1}}\notin{\rm
in}(I(\mathcal{L}))$, and  
\item[(iii)] $\widetilde{a}=\widetilde{\alpha}$.
\end{itemize}
\end{lemma}

\begin{proof} First we show the existence of $a$. If
$\alpha\in\mathcal{L}$, then $a=0$ satisfies (i), (ii) and (iii).
Assume that 
$\alpha\notin\mathcal{L}$. By
Lemma~\ref{may5-12}, $\widetilde{e}_i-\widetilde{e}_{s}$ is a
torsion element of $\mathbb{Z}^s/\mathcal{L}$ for $1\leq i\leq s-1$,
that is, there is a positive integer $n_i$ such that $n_i(e_i-e_s)$ is
in $\mathcal{L}$. If $\alpha_i$ is 
the $i^{\textup{th}}$ entry of $\alpha$, there 
are integers $q_i$ and $c_i$ such that $\alpha_i=q_in_i+c_i$ and
$0\leq c_i\leq n_i-1$. Hence, since $|\alpha|=0$, we can write
\begin{eqnarray*}
\alpha&=&\alpha_1(e_1-e_s)+\cdots+\alpha_{s-1}(e_{s-1}-e_s)\\
&=&c_1(e_1-e_s)+\cdots+c_{s-1}(e_{s-1}-e_s)+q_1n_1(e_1-e_s)+
\cdots+q_{s-1}n_{s-1}(e_{s-1}-e_s).
\end{eqnarray*}
If we set
$c=(c_1,\ldots,c_s)=c_1(e_1-e_s)+\cdots+c_{s-1}(e_{s-1}-e_s)$, 
then $\widetilde{c}=\widetilde{\alpha}$, $c\notin\mathcal{L}$ and
$|c|=0$. Consider the 
homogeneous binomial 
$$
f=t_1^{c_1}\cdots t_{s-1}^{c_{s-1}}-t_s^{-c_s}.
$$
Let $\mathcal{G}=\{g_1,\ldots,g_r\}$ be the reduced Gr\"obner
basis of $I(\mathcal{L})$, with respect to the revlex order, then
${\rm in}(I(\mathcal{L}))=({\rm in}(g_1),\ldots,{\rm in}(g_r))$. By
Remark~\ref{jun7-12}, $t_s$ does not divide any of the leading terms of
$g_1,\ldots,g_r$. Hence, by the division algorithm 
\cite[Theorem~3, p.~63]{CLO}, we can
write 
\begin{equation}\label{23-jul-10-1}
f=h_1g_1+\cdots+h_rg_r+g\tag{$*$}
\end{equation}
for some $h_1,\ldots,h_r$ in $S$, where $g=t_1^{b_1}\cdots
t_s^{b_s}-t_s^{-c_s}$ is homogeneous and $t^b=t_1^{b_1}\cdots t_s^{b_s}$ 
is not divisible by any of the leading
terms of $g_1,\ldots,g_r$, i.e., $t^b\notin {\rm
in}(I(\mathcal{L}))$. Thus, $t_1^{b_1}\cdots
t_{s-1}^{b_{s-1}}\notin{\rm in}(I(\mathcal{L}))$. Notice that $b_i>0$
for some $1\leq i\leq s-1$, otherwise $g=0$ and $c$ would
be in $\mathcal{L}$, a contradiction. By Eq.~(\ref{23-jul-10-1}), 
the binomial $f-g$ is in $I(\mathcal{L})$
and simplifies to
$$
f-g=t_1^{c_1}\cdots t_{s-1}^{c_{s-1}}-t_1^{b_1}\cdots t_s^{b_s}.
$$
Hence, by Lemma~\ref{aug-25-13}, 
$(c_1,\ldots,c_{s-1},0)-(b_1,\ldots,b_s)$ is in $\mathcal{L}$.
Consequently, one has  
\begin{equation}\label{may27-12}
(c_1,\ldots,c_{s-1},c_s)-(b_1,\ldots,b_{s-1},b_s+c_s)=
(c_1,\ldots,c_{s-1},0)-(b_1,\ldots,b_{s-1},b_s)\in\mathcal{L}.\tag{$**$}
\end{equation}
Consider the vector $a=(a_1,\ldots,a_s)$, where $a_i=b_i$ for
$i=1,\ldots,s-1$ and $a_{s}=b_s+c_s$. Then, by Eq.~(\ref{may27-12}),
$c-a\in\mathcal{L}$. 
Thus, $\widetilde{a}=\widetilde{c}$. 
For all the above, we get that $a$ 
satisfies (i), (ii) and (iii).

Next, we show the uniqueness of $a$. Assume that there are vectors
$a=(a_1,\ldots,a_s)$ and $a'=(a_1',\ldots,a_s')$ in 
$\mathbb{Z}^s$ that satisfy (i), (ii) and (iii). If $a_i\neq a_i'$ for
some $1\leq i\leq s-1$, then the binomial
$$
h=t_1^{a_1}\cdots t_{s-1}^{a_{s-1}}t_s^{-a_s}-t_1^{a_1'}\cdots
t_{s-1}^{a_{s-1}'}t_s^{-a_s'}
$$
is non-zero and belongs to $I(\mathcal{L})$ because
$a-a'\in\mathcal{L}$ by (iii), a contradiction because none of the two terms of
$h$ are in the initial ideal of $I(\mathcal{L})$ by (ii). Thus, $a_i=a_i'$ 
for $i=1,\ldots,s-1$. Since $|a|=|a'|$, we get $a=a'$.
\end{proof}

\begin{definition} An ideal $I\subset S$ is called a {\it complete
intersection\/} if  there exists $f_1,\ldots,f_{r}$ in $S $ such that
$I=(f_1,\ldots,f_{r})$,  
where $r$ is the height of $I$. 
\end{definition}

A graded ideal $I$ is a complete
intersection if and only if $I$ is generated by a homogeneous regular 
sequence with ${\rm ht}(I)$ elements (see \cite[Proposition~1.3.17,
Lemma~1.3.18]{monalg}). 

\begin{proposition}\label{may6-12-1} If $L\subset S$ is a graded lattice ideal of
dimension $1$, then there are positive integers
$n_1,\ldots,n_{s-1}$ such that 
\begin{itemize}
\item[(a)]$L'=(t_1^{n_1}-t_s^{n_1},\ldots,t_{s-1}^{n_{s-1}}-t_s^{n_{s-1}})
\subset L$, 
\item[(b)] ${\rm reg}(S/(t_s,L))\leq {\rm
reg}(S/(t_s,L'))=\sum_{i=1}^{s-1}(n_i-1)+1$, and
\item[(c)] $H_L(d)=H_L(d-1)=\deg\,
S/L$ for $d\geq\sum_{i=1}^{s-1}(n_i-1)+1$.
\end{itemize}
\end{proposition}

\begin{proof} (a): Let $\mathcal{L}\subset\mathbb{Z}^s$ be the lattice that
defines $L$, i.e., $L=I(\mathcal{L})$. By Lemma~\ref{may5-12}, there
are positive integers 
$n_1,\ldots,n_{s-1}$ such that $n_i(e_i-e_s)\in\mathcal{L}$ for all
$i$. Thus, $t_i^{n_i}-t_s^{n_i}\in L$ for all $i$. 

(b): Since $\dim(S/(t_s,L'))=\dim(S/(t_s,L))=0$, the Hilbert
polynomials of $S/(t_s,L')$ and $S/(t_s,L)$ are equal to
zero. Using the epimorphism 
$$
S/(t_s,L')\longrightarrow S/(t_s,L)\longrightarrow 0,
$$
we get that $H_{(t_s,L')}(d)\geq H_{(t_s,L)}(d)$ for $d\geq 0$. Hence,
the index of regularity of $S/(t_s,L)$ is bounded from above by the
index of regularity of $S/(t_s,L')$. The ideal
$I=(t_1^{n_1},\ldots,t_{s-1}^{n_{s-1}})$ is a complete intersection of
the polynomial ring $R=K[t_1,\ldots,t_{s-1}]$, hence the Hilbert
series of $R/I$ is equal to the polynomial  
$$
F_I(t)=(1+t+\cdots+t^{n_1-1})\cdots(1+t+\cdots+t^{n_{s-1}-1}),
$$
see \cite[p.~104]{monalg}. Thus, the index of regularity of $R/I$ is
$\sum_{i=1}^{s-1}(n_i-1)+1$ (see
Section~\ref{hilbert-function-section}). As $S/(t_s,L')\simeq R/I$, 
we get ${\rm reg}(S/(t_s,L'))=\sum_{i=1}^{s-1}(n_i-1)+1$. 

(c): Assume that $d\geq\sum_{i=1}^{s-1}(n_i-1)+1$. There is an exact
sequence of graded rings 
$$
0\longrightarrow (S/L)[-1]\stackrel{t_s}{\longrightarrow}
S/L\longrightarrow
S/(t_s,L)\longrightarrow 0.
$$
Hence, $H_L(d)-H_L(d-1)=
\dim_K(S/(t_s,L))_d=H_{(t_s,L)}(d)$. Therefore, 
using (b), we obtain $H_{(t_s,L)}(d)=0$ and $H_L(d)=H_L(d-1)=
{\rm deg}(S/L)$ (cf. Proposition~\ref{hilbert-function-dim=1}(ii)).  
\end{proof}

We come to the main result of this paper.

\begin{theorem}\label{degree-lattice} If $I(\mathcal{L})\subset S$ is
a graded lattice ideal of 
dimension $1$, then
$$
\deg\, S/I(\mathcal{L})=|T(\mathbb{Z}^s/\mathcal{L})|.
$$
\end{theorem}

\begin{proof} 
Let $\succ$ be the revlex order on the monomials of $S$ and let 
${\rm in}(I(\mathcal{L}))$ be the initial ideal of
$I(\mathcal{L})$. We set $d=\sum_{i=1}^{s-1}(n_i-1)+1$. By
Proposition~\ref{may6-12-1}, there are positive integers 
$n_1,\ldots,n_{s-1}$ such that $t_i^{n_i}-t_s^{n_i}\in
I(\mathcal{L})$ for all $i$ and $H_{I(\mathcal{L})}(d)=\deg\,
S/I(\mathcal{L})$. There is an injective map 
$$
\mathcal{B}_d=\{t^c\vert\, t^c\notin{\rm
in}(I(\mathcal{L}))\}\cap S_d\longrightarrow(S/I(\mathcal{L}))_d,\ \ \ \ 
t^c\mapsto t^c+I(\mathcal{L}).
$$
By a classical result in Gr\"obner bases theory \cite[Proposition 1,
p.~228]{CLO}, the image of this map is a basis for the $K$-vector space
$(S/I(\mathcal{L}))_d$. Thus, 
$|\mathcal{B}_d|=H_{I(\mathcal{L})}(d)$. Consider the map 
$$
\phi\colon\mathcal{B}_d\rightarrow T(\mathbb{Z}^s/\mathcal{L}), \ \ \
t^c=t_1^{c_1}\cdots 
t_s^{c_s}\stackrel{\phi}{\longmapsto}(c_1,\ldots,c_{s-1},c_s-d)+\mathcal{L}.
$$

The map $\phi$ is well defined, i.e., $\phi(t^c)$ is in
$T(\mathbb{Z}^s/\mathcal{L})$ for all $t^c$ in $\mathcal{B}_d$. 
This follows directly from Lemma~\ref{may5-12}(b) by noticing the
equality  
$$
(c_1,\ldots,c_{s-1},c_s-d)=c_1(e_1-e_s)+\cdots+c_{s-1}(e_{s-1}-e_s).
$$
 
Altogether, we need only show that $\phi$ is bijective. 
Notice that $t_s^d$ maps to $\widetilde{0}$ under $\phi$. 
By Lemma~\ref{may6-12}, the map $\phi$ is injective. To show that
$\phi$ is onto, take $\widetilde{a}\in T(\mathbb{Z}^s/\mathcal{L})$.
By Lemma~\ref{may6-12}, we may assume 
that $a_i\geq 0$ for $i=1,\ldots,s-1$ and $t_1^{a_1}\cdots
t_{s-1}^{a_{s-1}}\notin{\rm in}(I(\mathcal{L}))$. Notice that $0\leq
a_i\leq n_i-1$ for $i=1,\ldots,s-1$ because $t_i^{n_i}-t_s^{n_i}\in
I(\mathcal{L})$ for all 
$i$. Thus, $\sum_{i=1}^{s-1}a_i\leq\sum_{i=1}^{s-1}(n_i-1)<d$.
Consider the vector $c=(c_1,\ldots,c_s)$ given by $c_i=a_i$ for
$i=1,\ldots,s-1$ and $c_s=d-\sum_{i=1}^{s-1}a_i$. Then, the monomial
$t^c$ is in $\mathcal{B}_d$ and maps to $\widetilde{a}$ under the map
$\phi$.
\end{proof}

\begin{corollary}\label{may28-12}
Let $\mathcal{L}\subset\mathbb{Z}^s$ be a homogeneous lattice of rank $s-1$
generated as a $\mathbb{Z}$-module by the rows of an integral matrix
$A$. Then
$$
{\rm deg}\, S/I(\mathcal{L})=d_1\cdots d_{s-1},
$$
where $d_1,\ldots,d_{s-1}$ are the invariant factors of $A$.
\end{corollary}

\begin{proof} It is well known \cite[Theorem~II.9, pp.~26-27]{New}
that there are unimodular integral matrices $U$ and $V$ such that 
$$
UAV=D={\rm diag}\{d_1,\ldots,d_{s-1},0,\ldots,0\},
$$
$d_i>0$ for $1\leq i\leq s-1$ and $d_i$ divides $d_{i+1}$ for 
all $i$. In matrix theory terminology, this means that 
$D={\rm diag}\{d_1,\ldots,d_{s-1},0,\ldots,0\}$ is the 
{\it Smith normal form\/} of $A$ and $d_1,\ldots,d_{s-1}$ are the 
{\it invariant factors\/} of $A$. Hence, by the fundamental structure
theorem for finitely  
generated abelian groups \cite[pp.~187-188]{JacI}, we get 
$$
\mathbb{Z}^s/\mathcal{L}\simeq
\mathbb{Z}/(d_1)\oplus\cdots\oplus\mathbb{Z}/(d_{s-1})\oplus\mathbb{Z}\ 
\mbox{ and }\ T(\mathbb{Z}^s/\mathcal{L})\simeq
\mathbb{Z}/(d_1)\oplus\cdots\oplus\mathbb{Z}/(d_{s-1}).
$$
Thus, the result follows from Theorem~\ref{degree-lattice}. 
\end{proof}

\begin{corollary} Let $L\subset S$ be a graded lattice ideal of
dimension $1$. If $L$ is generated by the binomials 
$t^{\alpha_1^+}-t^{\alpha_1^-},\ldots,t^{\alpha_{m}^+}-t^{\alpha_{m}^-}$.
Then 
$$
{\rm deg}\, S/L=d_1\cdots d_{s-1},
$$
where $d_1,\ldots,d_{s-1}$ are the invariant factors of the matrix $A$
whose rows are $\alpha_1,\ldots,\alpha_m$.
\end{corollary}

\begin{proof} Let $\mathcal{L}$ be the homogeneous lattice that defines
the lattice ideal $L$. By Lemma~\ref{dec21-11}, one has the equality 
$\mathcal{L}=\mathbb{Z}\alpha_1+\cdots+\mathbb{Z}\alpha_m$.
Thus, the result follows at once from Corollary~\ref{may28-12}. 
\end{proof}

\begin{lemma}{\rm\cite[pp.~32-33]{stewart}}\label{stewart-tall} If
$H\subset G$ are free abelian groups of the same
rank $r$ with $\mathbb{Z}$-bases
$\delta_1,\ldots, \delta_r$ and $\gamma_1,\ldots, \gamma_r$ related by
$\delta_i=\sum_j g_{ij} \gamma_j$,
where $g_{ij}\in\mathbb{Z}$ for all $i,j$, then $|G/H|=|\det (g_{ij})|$.
\end{lemma}

\begin{definition}\rm
Let $\Delta$ be a {\it lattice\/} $n$-{\it simplex\/} in
$\mathbb{R}^s$, i.e., $\Delta$ is the convex hull of a set of
$n+1$ affinely independent points in $\mathbb{Z}^s$. The {\it
normalized volume\/} of $\Delta$ is defined as 
$n!{\rm vol}(\Delta)$. 
\end{definition}

The next result shows that the degree is the normalized volume 
of any $(s-1)$-simplex arising from a $\mathbb{Z}$-basis of
$\mathcal{L}$.

\begin{corollary}\label{jun2-12} If $\mathcal{L}\subset\mathbb{Z}^s$ is a
homogeneous lattice and $\alpha_1,\ldots,\alpha_{s-1}$
is a $\mathbb{Z}$-basis of $\mathcal{L}$, then 
$$
{\rm deg}\, S/I(\mathcal{L})=(s-1)!{\rm vol}({\rm
conv}(0,\alpha_1,\ldots,\alpha_{s-1})),
$$
where ${\rm vol}$ is the relative volume and 
${\rm conv}$ is the convex hull.
\end{corollary}

\begin{proof} By hypothesis,
$\mathcal{L}=\mathbb{Z}\alpha_1\oplus\cdots\oplus\mathbb{Z}\alpha_{s-1}$. 
Hence, using Lemma~\ref{may5-12}(b), we get the equality
$$ T(\mathbb{Z}^s/\mathcal{L})=
\mathbb{Z}(e_1-e_s)\oplus\cdots\oplus\mathbb{Z}(e_{s-1}-e_s)/
\mathbb{Z}\alpha_1\oplus\cdots\oplus\mathbb{Z}\alpha_{s-1}.
$$
For $1\leq i\leq s-1$, we can write 
$\alpha_i=\alpha_{i,1}(e_1-e_s)+\cdots+\alpha_{i,s-1}(e_{s-1}-e_s)$,
where $\alpha_{i,j}$ is the $j^\textup{th}$ entry of $\alpha_i$. Applying
Theorem~\ref{degree-lattice} and Lemma~\ref{stewart-tall} gives
$$
{\rm deg}\,
S/I(\mathcal{L})=|T(\mathbb{Z}^s/\mathcal{L})|=
\left|\det\left(   
\begin{array}{ccc}
\alpha_{1,1}&\cdots&\alpha_{1,s-1}\\
\vdots&\vdots&\vdots \\
\alpha_{s-1,1}&\ldots&\alpha_{s-1,s-1}
\end{array}
\right)\right| =(s-1)!{\rm
vol}(\Delta),
$$
where $\Delta={\rm
conv}(0,(\alpha_{1,1},\ldots,\alpha_{1,s-1}),\ldots,
(\alpha_{s-1,1},\ldots,\alpha_{s-1,s-1}))$
is a simplex in $\mathbb{R}^{s-1}$. To finish the proof we need only
show that ${\rm vol}(\Delta)={\rm vol}({\rm
conv}(0,\alpha_1,\ldots,\alpha_{s-1}))$. This follows from the very
definition of the notion of a 
relative volume (see \cite[Section~2]{ehrhart} and \cite[p.~238]{Sta5}).
\end{proof}

\begin{corollary}\label{jun2-12-bis} Let $I(\mathcal{L})\subset S$ be
a graded lattice ideal of 
dimension $1$. If $I(\mathcal{L})$ is a complete intersection
generated by
$t^{\alpha_1^+}-t^{\alpha_1^-},\ldots,t^{\alpha_{s-1}^+}-t^{\alpha_{s-1}^-}$,
then 
$$
{\rm deg}\, S/I(\mathcal{L})=(s-1)!{\rm vol}({\rm
conv}(0,\alpha_1,\ldots,\alpha_{s-1})).
$$
\end{corollary}

\begin{proof} By Lemma~\ref{dec21-11}, one has the equality 
$\mathcal{L}=\mathbb{Z}\alpha_1\oplus\cdots\oplus\mathbb{Z}\alpha_{s-1}$.
Thus, the formula for the degree follows from Corollary~\ref{jun2-12}.
\end{proof}

\begin{corollary} If $I(\mathcal{L})\subset S$ is a graded lattice ideal of 
dimension $1$, then $\mathbb{Z}^s/\mathcal{L}$ is torsion-free if 
and only if $I(\mathcal{L})=(t_1-t_s,\ldots,t_{s-1}-t_s)$.
\end{corollary}

\begin{proof} Assume that $\mathbb{Z}^s/\mathcal{L}$ is torsion-free.
Then, by Lemma~\ref{may5-12}(b), one has the equality.
$$
\mathcal{L}=\mathbb{Z}(e_1-e_s)\oplus\cdots\oplus\mathbb{Z}(e_{s-1}-e_s).
$$
Hence, $I(\mathcal{L})=(t_1-t_s,\ldots,t_{s-1}-t_s)$. The converse is
clear because the $(s-1)\times s$ matrix with rows
$e_1-e_s,\ldots,e_{s-1}-e_s$ diagonalizes over the integers to an
identity matrix. 
\end{proof}

\section{Computing some examples}\label{section-examples}

Given a set of generators of a homogeneous lattice
$\mathcal{L}\subset\mathbb{Z}^s$, a standard method to compute the
degree of the lattice ring $S/I(\mathcal{L})$ consists of two steps. 
First, one computes a generating set for $I(\mathcal{L})$ using the
following result:

\begin{proposition}{\cite[Lemma~7.6]{cca}}\label{saturation-gives-lattice} 
If $\mathcal{L}\subset\mathbb{Z}^s$ is a lattice generated by
$\alpha_1,\ldots,\alpha_m$ and $Q$ is the ideal generated by
$t^{\alpha_1^+}-t^{\alpha_1^-},\ldots,t^{\alpha_m^+}-t^{\alpha_m^-}$, 
then 
$$
(Q\colon(t_1\cdots t_s)^{\infty})=I(\mathcal{L}),
$$ 
where $(Q\colon h^{\infty}):=\{f\in S\vert\, fh^p\in Q\mbox{ for some
}p\geq 1\}$ is the saturation of $Q$ and $h=t_1\cdots t_s$. 
\end{proposition}

Second, one uses Hilbert functions (as described in
Section~\ref{hilbert-function-section}) to compute the degree of
$S/I(\mathcal{L})$. The handy command ``degree''  of {\it Macaulay}\/$2$
\cite{mac2} computes the degree. 

This standard method works for any homogeneous lattice. For
homogeneous lattices of rank $s-1$, our method is far more efficient,
especially with large examples. 

\begin{example}\label{jun2-12-1} Let $\mathcal{L}\subset\mathbb{Z}^5$
be the homogeneous lattice of rank 
$4$ generated by the rows
of the matrix 
$$
A=\left(
\begin{matrix}1001&  -500& -501&0&0\cr
0&     3500& -3500&0&0\cr 
0  & 0&3200& -200&-3000\cr
5000& -1000&-1000&-1001&-1999
\end{matrix}
 \right).
$$
The following procedure for {\it Maple\/} \cite{maple} 
\begin{verbatim}
with(LinearAlgebra):
A:=<1001,-500,-501,0,0; 0,3500,-3500,0,0; 
0,0,3200,-200,-3000; 5000,-1000,-1000,-1001,-1999>:
SmithForm(A);
\end{verbatim}
computes the Smith normal form of $A$:
$$
D=\left(
\begin{matrix}1& 0& 0&0&0\cr
0&  1&0&0&0\cr 
0  & 0&100&0&0\cr
0&0&0&91203112000&0
\end{matrix}
 \right).
$$
Thus, by Theorem~\ref{degree-lattice}, we obtain $\deg\,
S/I(\mathcal{L})=(2^8)(5^5)(7^2)(11)(13)(1627)$. 
The standard procedure for computing the degree of $S/I(\mathcal{L})$ fails for this
example. Indeed, {\it Macaulay\/}$2$ does not even compute the
saturation $(Q\colon h^\infty)$ of the ideal 
$$
Q=(t_1^{1001}-t_2^{500}t_3^{501},t_2^{3500}-t_3^{3500},
t_3^{3200}-t_4^{200}t_5^{3000},
t_1^{5000}-t_2^{1000}t_3^{1000}t_4^{1001}t_5^{1999})
$$
with respect to $h=t_1t_2t_3t_4t_5$. Notice that $Q$ is a complete
intersection and accordingly 
$$\deg(S/Q)=(1001)(3500)(3200)(5000)=
(2^{12})(5^9)(7^2)(11)(13).$$
\end{example}

\begin{remark}\label{may31-12} Given an integral matrix $A$, the {\it Macaulay}$2$
function ``smithNormalForm'' produces a diagonal
matrix $D$, and unimodular matrices $U$ and $V$ such that 
$D =UAV$. Warning: even though this function is called the
      Smith normal form, it doesn't necessarily satisfy the more stringent
      condition that the diagonal entries $d_1, d_2,\ldots,d_m$ of
      $D$ satisfy: $d_1\vert d_2\vert\cdots\vert d_m$. For this
      reason we prefer to use {\it Maple\/} \cite{maple} to compute
      the Smith normal form of $A$.
\end{remark}

\begin{example}\label{jun2-12-2} Let $\mathcal{L}\subset\mathbb{Z}^3$
be the homogeneous lattice of rank 
$2$ generated by the rows
of the matrix 
$$
A=\left(
\begin{matrix} 18& -18& 0\cr
45&  0& -45\cr 
0  &10&-10\cr
\end{matrix}
 \right).
$$
The following procedure for {\it Maple\/} \cite{maple} 
\begin{verbatim}
with(LinearAlgebra):
A:=<18,-18,0; 45,0,-45; 0,10,-10>:
SmithForm(A);
\end{verbatim}
computes the Smith normal form of $A$:
$$
D=\left(
\begin{matrix}
1&0&0\cr 
0  &90&0\cr
0&0&0
\end{matrix}
 \right)
$$
Thus, by Theorem~\ref{degree-lattice}, we obtain $\deg\,
S/I(\mathcal{L})=90$. The standard procedure for computing the degree
of $S/I(\mathcal{L})$ works fine in this ``small'' example. Indeed,
using the following procedure for {\it Macaulay\/}$2$ 
\begin{verbatim}
S=QQ[t1,t2,t3]
Q=ideal(t1^18-t2^18,t1^45-t3^45,t2^10-t3^10)       
saturate(Q,t1*t2*t3)
degree saturate(Q,t1*t2*t3)
\end{verbatim}
we obtain 
$$
I(\mathcal{L})=(Q\colon(t_1t_2t_3)^\infty)=(t_1^9-t_2^4t_3^5,t_2^{10}-t_3^{10})
\ \mbox{ and }\ \deg(S/I(\mathcal{L}))=90.
$$
\end{example}

\begin{remark} The program {\it Normaliz} \cite{normaliz2} computes
the {\it normalized volume\/} of lattice polytopes. Hence, 
by Corollary~\ref{jun2-12}, we can use this
program with the handy option  {\tt -v} to compute the degree. This of
course requires computing a $\mathbb{Z}$-basis of the lattice first. We
computed the degree of Example~\ref{jun2-12-2} without any problem 
using ``{\tt normbig.exe}''. 
\end{remark}

Our main result does not extends to graded lattice ideals 
of dimension $\geq 2$.

\begin{example}\label{jun4-12} Consider the homogeneous lattice
$\mathcal{L}=\mathbb{Z}(-1,2,-1)\subset\mathbb{Z}^3$. Then,
$$
I(\mathcal{L})=(t_2^2-t_1t_3)\ \mbox{ and }\ \deg\,
\mathbb{Q}[t_1,t_2,,t_3]/I(\mathcal{L})=2\neq
1=|T(\mathbb{Z}^3/\mathcal{L})|.
$$
\end{example}

\section{Vanishing ideals over finite
fields}\label{vanishing-id-section}

In this section, we link our results
to vanishing ideals over finite fields and 
present an application. 
Vanishing ideals are connected to coding theory as is seen
below. 

\begin{definition}
The {\it projective space\/} of 
dimension $s-1$ over a field $K$, denoted by 
$\mathbb{P}^{s-1}$, is the quotient space 
$$(K^{s}\setminus\{0\})/\sim $$
where two points $\alpha$, $\beta$ in $K^{s}\setminus\{0\}$ 
are equivalent, $\alpha \sim \beta$, if $\alpha=\lambda{\beta}$ for
some $\lambda\in K$. We 
denote the  
equivalence class of $\alpha$ by $[\alpha]$. 
\end{definition}

Let $\mathbb{F}_q$  be a finite field with $q$ elements and 
let $v_1,\ldots,v_s$ be a sequence of vectors in
$\mathbb{N}^n$ with $v_i=(v_{i1},\ldots,v_{in})$ for $1\leq i\leq
s$. 
Consider the {\it projective algebraic toric
set\/} 
$$
X:=\{[(x_1^{v_{11}}\cdots x_n^{v_{1n}},\ldots,x_1^{v_{s1}}\cdots
x_n^{v_{sn}})]\, \vert\, x_i\in \mathbb{F}_q^*\mbox{ for all
}i\}\subset\mathbb{P}^{s-1}
$$
parameterized by the monomials 
$x^{v_1},\ldots,x^{v_s}$, where $\mathbb{F}_q^*=\mathbb{F}_q\setminus\{0\}$ and 
$\mathbb{P}^{s-1}$ is the projective space of dimension $s-1$ over
the field $\mathbb{F}_q$. The set 
$X$ is a multiplicative group under componentwise multiplication. 

Let $S=\mathbb{F}_q[t_1,\ldots,t_s]=S_0\oplus S_1\oplus \cdots\oplus
S_d\oplus\cdots$ be a
polynomial ring over the field $\mathbb{F}_q$ with the standard
grading. 
Recall that the {\it vanishing ideal\/} of $X$, denoted by
$I(X)$, is the ideal 
of $S$ generated by the homogeneous polynomials that vanish at all
points of $X$. 

According to the next theorem, our results can be applied to this
family of vanishing ideals.

\begin{theorem}\label{lemma-homogeneous-lattice} If $\mathbb{F}_q$ is a finite
field, then 
\begin{itemize}
\item[(a)] \cite{geramita-cayley-bacharach} $I(X)$ is a radical $1$-dimensional
Cohen-Macaulay ideal.

\item[(b)] \cite{algcodes} There is a unique homogeneous
lattice $\mathcal{L}$ such that $I(X)=I(\mathcal{L})$. 

\item[(c)] \cite[Lecture 13]{harris} $H_{I(X)}(d)=|X|$ for $d\geq |X|-1$.
\end{itemize}
\end{theorem}

Hence, by (c), the degree of $S/I(X)$ is equal to $|X|$. Thus, our
results can be used to 
compute $|X|$, especially in cases where the homogeneous lattice that 
defines the ideal $I(X)$ is known (see for instance
\cite[Theorem~2.5]{algcodes} for such cases). 

The degree of $S/I(X)$ is relevant from the
viewpoint of algebraic coding theory as we now briefly explain. 
Roughly speaking, an {\it
evaluation code\/} over $X$ of degree $d$ is a linear space obtained
by evaluating all homogeneous $d$-forms of $S$ on
the set of points $X\subset{\mathbb P}^{s-1}$ (see
\cite{duursma-renteria-tapia,gold-little-schenck}). 
An evaluation code over
$X$ of degree $d$ has {\it length\/} $|X|$ and {\it dimension\/}
$H_{I(X)}(d)$. The main parameters (length, dimension, 
minimum distance) of evaluation codes of this type have been studied in
\cite{duursma-renteria-tapia,gold-little-schenck,hansen,cartesian-codes,ci-codes}.
The index of regularity is useful in coding theory
 because potentially good evaluation codes can occur only 
if $1\leq d < {\rm reg}(S/I(X))$. 

The complete intersection property of $I(X)$ was recently
characterized in \cite{ci-lattice} in algebraic and geometric terms
(see also \cite{d-compl}). 
If $X$ is parameterized by the edges of a clutter, then $I(X)$ is a
complete intersection if and only if $X$ is a projective torus 
\cite{ci-codes}.

Let $\mathcal{L}$ be the homogeneous lattice that defines $I(X)$. The
next result shows how the algebraic structure of $\mathbb{Z}^s/\mathcal{L}$ is
reflected in the algebraic structure of $I(X)$.

\begin{corollary}\label{jun2-12-3} If $q-1$ is a prime number such that 
$v_i\not\equiv v_j\! \mod(q-1)$ for $i\neq j$ and 
$T(\mathbb{Z}^s/\mathcal{L})\simeq(\mathbb{Z}_{q-1})^{s-1}$, then $I(X)$ is a complete
intersection if and only if  
$$I(X)=(t_1^{q-1}-t_s^{q-1},\ldots,t_{s-1}^{q-1}-t_s^{q-1}).$$
\end{corollary}

\begin{proof} Assume that $I(X)$ is a complete intersection, i.e., the
ideal $I(X)$ is generated by homogeneous binomials
$f_1,\ldots,f_{s-1}$ of degrees $\delta_1,\ldots,\delta_{s-1}$. 
The linear binomial $t_i-t_j$ is not in $I(X)$ for any
$i\neq j$, this follows using that $v_i\not\equiv v_j\! \mod(q-1)$.
Thus, $\deg(f_i)=\delta_i\geq 2$ for all $i$. By Theorem~\ref{degree-lattice},
we have 
$$
{\rm deg}\, S/I(X)=(q-1)^{s-1}=\delta_1\cdots \delta_{s-1}.
$$
As $q-1$ is prime, we get that $\delta_i=q-1$ for all $i$. Consider
the $\mathbb{F}_q$-vector spaces
$$V=\mathbb{F}_q(t_1^{q-1}-t_s^{q-1})+\cdots+\mathbb{F}_q(t_{s-1}^{q-1}-t_s^{q-1})
\ \mbox{ and }\ I(X)_{q-1}=\mathbb{F}_qf_1+\cdots+\mathbb{F}_qf_{s-1}.
$$
It suffices to show the equality $V=I(X)_{q-1}$. 
Since $t_i^{q-1}-t_s^{q-1}$ vanishes at all points of $X$ for all $i$, we get 
that $t_i^{q-1}-t_s^{q-1}\in I(X)_{q-1}$ for all $i$. 
Consequently $V=I(X)_{q-1}$ because $V$ and $I(X)_{q-1}$ have the same
dimension. The converse is clear because
$t_1^{q-1}-t_s^{q-1},\ldots,t_{s-1}^{q-1}-t_s^{q-1}$ form a regular
sequence and the height of $I(X)$ is $s-1$.
\end{proof}

\section{Problems and related results}\label{future-works-section}

Let $G$ be a connected simple graph. Two prime examples of graded
lattice ideals 
of dimension one 
are the vanishing ideal $I(X)$ of a set $X$ parameterized by the edges of
$G$ \cite{algcodes} and the toppling ideal $I_G$ of $G$
\cite{perkinson}. 

If $\mathcal{L}$ is the lattice defining $I(X)$ or $I_G$, 
it would be interesting to understand how 
the group structure, as a finite
abelian group, of $T(\mathbb{Z}^s/\mathcal{L})$ is reflected in
the algebraic structure of $I(X)$ and $I_G$ (cf.
Corollary~\ref{jun2-12-3}), and viceversa. It would also be
interesting to compute the index of  
regularity using linear algebra methods. The index of regularity of
$I(X)$ is
used in algebraic coding theory to find potentially good parameterized
codes \cite{cartesian-codes,evencycles}.  

A  difficult problem is to determine the group structure
 of $T(\mathbb{Z}^s/\mathcal{L})$ in terms of the 
combinatorics of $G$. In the setting of algebraic graph theory
\cite{alfaro-valencia,lorenzini}, 
 the group $T(\mathbb{Z}^s/\mathcal{L})$ associated to $I_G$ is called
the sandpile group of $G$. Its group structure 
is only known for a few families of graphs (see \cite{alfaro-valencia,lorenzini} 
and the references therein). By the Kirchhoff's matrix tree theorem, 
the order of $T(\mathbb{Z}^s/\mathcal{L})$  is the number of spanning
trees of $G$ (see \cite[Theorem~6.3, p.~39]{Biggs} and
\cite[Theorem~1.1]{laplacian-survey}). Thus, by our main 
result, one obtains a nice combinatorial formula for the degree of $S/I_G$.
On the other hand in the setting of commutative algebra and coding
theory \cite{evencycles,algcodes}, the 
degree of $S/I(X)$ has been computed in terms of the combinatorics of
$G$ \cite[Theorem~3.2]{evencycles}. 
Thus, our main result gives a combinatorial formula for the order of
$T(\mathbb{Z}^s/\mathcal{L})$.    

\bigskip

\noindent
{\bf Acknowledgments.} 
The authors would like to thank the referees for their
careful reading of the paper and for the improvements that
they suggested.

\bibliographystyle{plain}

\end{document}